\documentclass[12pt]{amsart}
\usepackage{amssymb,latexsym,amsmath,enumerate,nicefrac}
\usepackage{amsaddr}

\usepackage[total={6.5in,8.5in}, top=1.25in, left=1in]{geometry}

\usepackage{amsthm}

\numberwithin{equation}{section}

\usepackage{color}
\definecolor{myurlcolor}{rgb}{0.6,0,0}
\definecolor{mycitecolor}{rgb}{0,0,0.8}
\definecolor{myrefcolor}{rgb}{0,0,0.8}
\usepackage[bookmarks=false,pagebackref=true]{hyperref}
\hypersetup{colorlinks,
linkcolor=myrefcolor,
citecolor=mycitecolor,
urlcolor=myurlcolor}

\renewcommand*{\backref}[1]{}

\newtheorem{theorem}{Theorem}[section]
\newtheorem*{Thm:Goal}{Theorem \ref{Thm:Goal}}

\newtheorem{proposition}[theorem]{Proposition}
\newtheorem*{corollary}{Corollary}

\theoremstyle{definition}
\newtheorem{definition}[theorem]{Definition}

\theoremstyle{remark}

\usepackage[all]{xy}
\usepackage{tikz}
\usetikzlibrary{calc}
\usepackage{graphics}
\usepackage{dsfont}
\usepackage{amsfonts}
\usepackage{mathtools}
\usepackage{comment}
\usetikzlibrary{arrows}
 \usetikzlibrary{matrix,positioning}

\usepackage{amssymb}
\usepackage{mathrsfs}
\usepackage{comment}

\allowdisplaybreaks

\DeclareMathOperator*{\Lo}{o}
\DeclareMathOperator*{\Circ}{Circ}

\makeatletter
\renewcommand*\env@matrix[1][\arraystretch]{%
  \edef\arraystretch{#1}%
  \hskip -\arraycolsep
  \let\@ifnextchar\new@ifnextchar
  \array{*\c@MaxMatrixCols c}}
\makeatother

\includecomment{pf}

\excludecomment{expo}

\excludecomment{cut}

\makeatletter
\newcommand*\bigcdot{\mathpalette\bigcdot@{.6}}
\newcommand*\bigcdot@[2]{\mathbin{\vcenter{\hbox{\scalebox{#2}{$\m@th#1\bullet$}}}}}
\makeatother

\newcounter{BWYtable}

\newcounter{BWYDiagram}

\newcounter{BWYFigure}

\newcommand{\w}{{\tfrac{w}{r}}}
\newcommand{\cll}{{\tfrac{L}{r}}}

\usetikzlibrary{decorations.pathmorphing}

\pgfkeys{/tikz/.cd, arrowhead/.default=-1pt, arrowhead/.code={}}

\begin{document}

\title{Buffon's Problem Determines Gaussian Curvature In Three Geometries}%

\author{Aizelle Abelgas$^\ast$, Bryan Carrillo$^\dagger$, John Palacios$^\ast$, David Weisbart$^\ast$, \and Adam M. Yassine$^\ddag$}
\address{
\begin{tabular}[h]{cc}
 $^\ast$Department of Mathematics &  $^\dagger$Department of Mathematics\\
  University of California, Riverside &  Saddleback College
  \end{tabular}}%
\address{$^\ddag$Department of Mathematics and Statistics\\Pomona College}%
\email{aabel003@ucr.edu} \email{bcarrillo@saddleback.edu} \email{jpala019@ucr.edu}\email{weisbart@math.ucr.edu}\email{adam.yassine@pomona.edu}

\pagestyle{plain}

\begin{abstract} 
A version of the classical Buffon problem in the plane naturally extends to the setting of any Riemannian surface with constant Gaussian curvature.  The Buffon probability determines a Buffon deficit.  The relationship between Gaussian curvature and the Buffon deficit is similar to the relationship that the Bertrand-Diguet-Puiseux Theorem establishes between Gaussian curvature and both circumference and area deficits.
\end{abstract}


\maketitle

\tableofcontents


\section{Introduction}\label{sec:intro}

Buffon's problem has attracted considerable interest and invited many generalizations and analogues following its introduction \cite{Buff1, Buff2}.  Barbier studied an analogue of this problem in the planar setting with the ``needle'' replaced by a ``noodle'', a rigid rectifiable plane curve \cite{Bar}.  Diaconis studied the problem in the setting of a long needle \cite{Dia}, estimated the moments of the random variable that counts the number of intersections, and discussed an applied motivation for this problem in detection deployment.  Solomon's extensive treatment and review of the problem \cite{Sol} presents many of its generalizations.  Klain and Rota treat the problem in the planar setting by relating it to the Crofton formula \cite{Rota}.  Peter and Tanasi \cite{PT} and Isokawa \cite{Iso} studied analogues of Buffon's problem on the sphere, the former with gratings formed by lines of longitude, and the latter with gratings formed by lines of latitude.  

The Crofton formula is closely related to Buffon's problem and the literature on the development and application of this formula in the study of differential geometry is vast.  Calegari's lively article \cite{Calegari:Notices:2020} is an engaging starting point for a reader who is interested in further study.  Buffon's problem may be extended to the setting of certain planar fractals.   A theorem of Besicovitch \cite[Theorem 6.13]{Falconer:Book:1986} implies that the intersection of a dropped needle with a fractal that belongs to a certain large class of planar fractals is a zero probability event \cite{Peres_Solomyak:PJM:2002}.  The Favard length of a subset $E$ of the unit square is proportional to the probability that a needle dropped in the square intersects $E$.  It appears that Peres and Solomyak \cite{Peres_Solomyak:PJM:2002} were the first to study the problem of estimating the decay of the Favard length for a class of self similar sets that includes $K^2$, the Cartesian product of the middle half Cantor set $K$ with itself.  Although a survey of the literature is beyond the scope of this work, it is important for perspective to at least recognize some appearances of ideas that are related to Buffon's problem in both the fractal and differential geometric settings.

It seems worthwhile to explore a connection between Buffon's problem and Gaussian curvature that utilizes only elementary prerequisite knowledge.  Take $M$ to be any Riemannian surface with constant Gaussian curvature and denote by $\kappa(M)$ the Gaussian curvature of $M$.  Section~\ref{genframe} describes a procedure for dropping a \emph{geodesic needle} of length $2L$ on $M$ and defines a Bernoulli random variable $I(X)$ that takes on the value 1 if the needle intersects a certain grating in $M$ with spacing $2L$.  The \emph{Buffon probability} is the probability $P(I(X)=1)$.  The principle novelty of this paper is the introduction of the notion of a \emph{Buffon deficit} together with Theorem~\ref{Thm:Goal}, which establishes a probabilistic characterization of Gaussian curvature.

\begin{definition}
The \emph{Buffon deficit} for $M$ for a needle of length $2L$ is the difference \[\frac{2}{\pi} - P(I(X)=1).\]
\end{definition}

\begin{theorem}\label{Thm:Goal}
For any Riemannian surface $M$ with constant Gaussian curvature, \[\lim_{L\to 0^+}\frac{9\pi}{2}\frac{P(I(X)=1) - \frac{2}{\pi}}{L^2} = \kappa(M).\]
\end{theorem}

\begin{corollary}
For any Riemannian surface $M$ with constant Gaussian curvature, \[\lim_{L\to 0^+} P(I(X)=1) = \frac{2}{\pi}.\]
\end{corollary}

\noindent The relationship between Buffon deficits and Gaussian curvature that Theorem~\ref{Thm:Goal} establishes is similar to the relationship that the Bertrand-Diguet-Puiseux Theorem \cite{BDP} establishes for circumference and area deficits more generally for any Riemannian surface.

 %


\section{A General Framework}\label{genframe}
 
\subsection{Isometries and gratings} Any Riemannian surface with constant Gaussian curvature is isometric to one of these surfaces: the plane, $\mathds R^2$; the sphere of radius $r$, $\mathds{S}_r$; or for any positive real number $k$, the Poincar\'{e} disk with its usual metric scaled by a factor of $k^2$, $\mathds{H}_k$.  Take $M$ to be any one of these surfaces.  The group $\mathcal G$ of isometries of $M$ acts transitively on $M$.  For any $g$ in $\mathcal G$ and $x$ in $M$, denote by $gx$ the result of applying the transformation $g$ to the point $x$.  For any subset $S$ of $M$, denote by $gS$ the set \[gS = \{gs\colon s\in S\}.\]  Henceforth, take $L$ to be any positive real number if $M$ is $\mathds R^2$ or $\mathds H_k$, and if $M$ is $\mathds S_r$, take $L$ to be equal to $\tfrac{\pi r}{2n}$ for any natural number $n$ that is greater than 1.

\begin{definition}
An \emph{equator for $M$} is a directed geodesic in $M$.  A \emph{grating line} for an equator $\mathcal E$ is a geodesic that intersects $\mathcal E$  at a right angle.  A \emph{grating $G(L)$ for $M$ with equator $\mathcal E$ and spacing $2L$} is a set of grating lines for $\mathcal E$ so that $G(L)\cap \mathcal E$ is an evenly spaced set of points with smallest spacing equal to $2L$.
\end{definition}

Take $\mathcal G^\prime$ to be the maximal nontrivial subgroup of $\mathcal G$ that preserves $\mathcal E$, the direction of $\mathcal E$, and acts transitively on $\mathcal E$.  There is an $h_0$ in $\mathcal G^\prime$ that moves points in $\mathcal E$ in the positive direction along $\mathcal E$ and that generates the subgroup $\mathcal H$ of $\mathcal G^\prime$ that both preserves and acts transitively on $G(L)\cap \mathcal E$. For any $g$ in $\mathcal G^\prime$, denote by $\alpha_g$ the signed distance that $g$ moves points along $\mathcal E$.  The \emph{displacement function}, $\alpha$, takes each $g$ in $\mathcal G^\prime$ to $\alpha_g$.  In the case of $\mathds S_r$, the range of $\alpha$ is $\mathds R \bmod 2\pi r$.  The displacement function is a homomorphism that orders the elements of $\mathcal G^\prime$.  Furthermore, for any $g$ in $\mathcal G^\prime$, \[\alpha_{g^{-1}} = -\alpha_g \quad \text{and} \quad \alpha_{h_0} = 2L.\]

\begin{proposition}\label{prop:2:EtoEandGtoG}
For any $M$, any two equators $\mathcal E_1$ and $\mathcal E_2$ in $M$, and any two gratings $G_1(L)$ and $G_2(L)$ with common spacing $2L$ and respective equators $\mathcal E_1$ and $\mathcal E_2$, there is a $g$ in $\mathcal G$ so that \[g\mathcal E_1 = \mathcal E_2 \quad \text{and} \quad gG_1(L) = G_2(L).\]
\end{proposition}

 
 \subsection{Dropping the needle}  For each of the three types of surfaces, view a needle as a directed segment of length $2L$ of a geodesic.  Take $X$ to be the random variable that is uniformly distributed in $[-L, L]$.  For any real number $z$ in the case when $M$ is $\mathds R^2$ or $\mathds H_k$, or for any $z$ in $\mathds R\bmod 2\pi r$ in the case when $M$ is $\mathds S_r$, and for any $x$ in $\mathcal E$, denote by $p_x(z)$ the unique point in $\mathcal E$ that is a signed distance of $z$ from $x$, like this:

\bigskip

\begin{center}
 \noindent\begin{tikzpicture}[scale = .9]
 \draw[] (-5,0) -- (4.95,0);
 \draw[very thin, >=triangle 45, ->] (-5,0) -- (5,0);
 \node[font = \small] at (-4.75, -.25) {$\mathcal E$};
 
 \draw[very thick] (-1.75,0) -- (2.25,0);

 \draw[] (0,-2.5) -- (0,2.5);

 \draw[] (-4, 2.75) -- (-4, 2.95) -- (-2, 2.95) -- (0, 2.95) -- (0, 2.75);
 \node[outer sep = 0pt, inner sep = 2pt, font = \small, fill = white] at (-2,2.95) {$2L$};
 
 \draw[] (-4,-2.5) -- (-4,2.5);
 \draw[] (4,-2.5) -- (4,2.5);
 
 \draw[] (-1,0) circle (2);
 \node[font = \small] at (-3.35, -1) {$C_x(z)$};

 \draw[fill = black] (-1.75,0) circle (1.5pt);
 \node[font = \small] at (-1.75, .5) {$p_x(-L)$};
 \draw[fill = black] (2.25,0) circle (1.5pt);
 \node[font = \small] at (2.25, .5) {$p_x(L)$};

 \draw[fill = black] (.25,0) circle (2pt);
 \node[font = \small] at (.25, .35) {$x$};

 \draw[fill = black] (-1,0) circle (3pt);
 \draw[thick, dotted] (-1,0) -- (0,1.732);
 \draw[thick, dotted] (-1,0) -- (0,-1.732);
 
 \tikzset{decoration={snake,amplitude=1.5pt,segment length=2pt,
                       post length=0mm,pre length=0mm}}
 \draw[decorate, thick] (0,-1.732) arc (-60:60:2);
 
 \draw[fill = black] (0,1.732) circle (2pt);
 \draw[fill = black] (0,-1.732) circle (2pt);

 \node[fill = white, inner sep = 0pt, font = \small] at (-1, -.5) {$p_x(z)$};

 \begin{scope}[xshift = -.395in]
 \draw[] (60:2.2) -- (60:2.6);
 \draw[] (-60:2.2) -- (-60:2.6);
 \draw[] (30:2.4) -- (30:2.8);
 \node[font = \small] at (30:3.3) {$\Lambda_x(z)$};
 \end{scope}
 \draw[] (.2,-2.078) arc (-60:60:2.4);

\end{tikzpicture}
\end{center}

\bigskip

The tip of the needle with center $p_x(z)$ is a marked endpoint of the needle and its position is uniformly randomly distributed on the geodesic circle $C_x(z)$ of radius $L$ and center $p_x(z)$.  The geodesic circle $C_x(z)$ intersects either one grating line at two distinct points or two distinct grating lines at one point each.  In the case when $C_x(z)$ intersects a geodesic in $G(L)$ at two points, the intersection defines two arcs of $C_x(z)$.  Denote by $\Lambda_x(z)$ the arc length of the smaller of the two arcs. If the geodesic circle $C_x(z)$ intersects two grating lines at exactly one point each, then define $\Lambda_x(z)$ to be 0. If $C_x(z)$ intersects a grating line in such a way that the intersection divides $C_x(z)$ into two arcs of equal length, then take $\Lambda_x(z)$ to be half the circumference of $C_x(z)$.  

Denote by $\Circ(L)$ the circumference of a geodesic circle of radius $L$.  Denote by $I_x(X)$ the random variable that is 1 if a needle with midpoint $p_x(X)$ intersects $G(L)$ and $0$ otherwise.  A needle that intersects a grating line will still intersect the same grating line at the same point when the needle is rotated by half of a circle, hence
\begin{equation}\label{1.1}
P(I_x(X)=1|X=z) = \frac{2\Lambda_x(z)}{\Circ(L)}.
\end{equation}
The law of total probability implies that
\begin{align}\label{1.3}
P(I_x(X)=1) 
= \frac{1}{L}\int_{-L}^{L} \frac{\Lambda_x(z)}{\Circ(L)}\,{\rm d}z.
\end{align}

\subsection{Consequences of homogeneity}

The goal of this subsection is to show that the probability of intersection of a dropped needle is independent of the choices of $\mathcal E$, $G(L)$, and $x$.   To see this, take $x_0$ to be any element of $G(L)\cap\mathcal E$.   

\begin{proposition}\label{invarianceingeneral}
For any $g$ in $\mathcal G^\prime$, \[P(I_{g x_0}(X) = 1) = P(I_{x_0}(X) = 1).\]
\end{proposition}

\begin{proof}
Since any $h$ in $\mathcal H$ preserves intersections, arc lengths, and the grating $G(L)$,
\begin{equation}\label{HPreservesIntLength}\Lambda_{hx_0}(z) = \Lambda_{x_0}(z).
\end{equation} This invariance under $h$ together with \eqref{1.3} implies that \begin{equation}\label{Lem:ginHsymmetryB}
P(I_{hx_0}(X)=1) = P(I_{x_0}(X)=1).\end{equation}  For any $z$ in $[-L, L]$, \eqref{1.3} together with the equality \begin{equation}\label{eq:A:shift}\Lambda_{gx_0}(z) = \Lambda_{x_0}(z + \alpha_g)\end{equation} and a change of variables implies that
\begin{align}\label{eq:2:inttobesplit}
P(I_{gx_0}(X)=1) &=\frac{1}{L}\int_{-L+\alpha_{g}}^{L+\alpha_g} \frac{\Lambda_{x_0}(z)}{\Circ(L)}\,{\rm d}z.
\end{align}
If $\alpha_{g}$ is in $[0, L]$, then
\begin{align*}
P(I_{gx_0}(X)=1) &= \frac{1}{L}\int_{-L+\alpha_{g}}^{L} \frac{\Lambda_{x_0}(z)}{\Circ(L)}\,{\rm d}z+ \frac{1}{L}\int_{L}^{L+\alpha_{g}} \frac{\Lambda_{x_0}(z)}{\Circ(L)}\,{\rm d}z.
\end{align*}
A change of variables together with \eqref{HPreservesIntLength} and \eqref{eq:A:shift} implies that
\[
\frac{1}{L}\int_{L}^{L+\alpha_{g}} \frac{\Lambda_{x_0}(z)}{\Circ(L)}\,{\rm d}z = \frac{1}{L}\int_{L}^{L+\alpha_{g}} \frac{\Lambda_{h^{-1}_0x_0}(z)}{\Circ(L)}\,{\rm d}z = \frac{1}{L}\int_{L-2L}^{L+\alpha_{g}-2L} \frac{\Lambda_{x_0}(z)}{\Circ(L)}\,{\rm d}z,\]
and so
\begin{align}\label{Lem:ginGsmallsymmetry}
P(I_{gx_0}(X)=1) = P(I_{x_0}(X)=1).
\end{align}  
For any $g$ in $\mathcal G^\prime$, there is an $h$ in $\mathcal H$ and a $g^\prime$ in $\mathcal G^\prime$ so that $\alpha_{g^\prime}$ is in $[0, L]$ and \begin{equation}\label{EQ:Consequence:gequalshgprime}hg^\prime = g.\end{equation}  Associativity of the group action together with \eqref{Lem:ginHsymmetryB}, \eqref{Lem:ginGsmallsymmetry}, and \eqref{EQ:Consequence:gequalshgprime} extends  \eqref{Lem:ginGsmallsymmetry} to all $g$ in $\mathcal G^\prime$.
\end{proof}

For any grating $G_1(L)$, denote by $\mathcal E_1$ the equator for $G_1(L)$.  Any $x_1$ in $\mathcal E_1$ is the center of a unique segment of $\mathcal E_1$ of length $2L$.  Take $P_1(I_{x_1}(X) = 1)$ to be the probability that a dropped needle of length $2L$ intersects $G_1(L)$, where the center of the needle is in the segment of $\mathcal E_1$ that $x_1$ defines, and is uniformly randomly distributed with respect to arc length on that segment.

\begin{theorem}\label{Theorem:Consequence:invariance}
For any grating $G_1(L)$ and any $x_1$ in $\mathcal E_1$, \[P_1(I_{x_1}(X)=1) = P(I_{x_0}(X) = 1).\]
\end{theorem}

\begin{proof}
Proposition~\ref{prop:2:EtoEandGtoG} guarantees that there is a $g_1$ in $\mathcal G$ that takes ${\mathcal E}_1$ to ${\mathcal E}$ and $G_1(L)$ to $G(L)$.  Isometries preserve intersections and the lengths of arcs, so \[P_1(I_{x_1}(X) = 1) = P(I_{g_1x_1}(X)=1).\]  Since $g_1x_1$ is in $\mathcal E$, there is a $g^\prime_1$ in $\mathcal G^\prime$ so that $g^\prime_1 x_0$ is equal to $g_1x_1$, and so Proposition~\ref{invarianceingeneral} implies the desired equality.
\end{proof}


\section{Determination of the Buffon Probabilities}

\subsection{The needle in the plane}

Rotations around the origin, reflection across the $y$-axis, and translations generate $\mathcal G$, the isometry group of the plane.  The geodesics are the straight lines and the circumference of any circle of radius $L$ is equal to $2\pi L$.  Take the equator $\mathcal E$, the subgroup $\mathcal G^\prime$ of $\mathcal G$, and the subgroup $\mathcal H$ of $\mathcal G^\prime$ to be, respectively, the positively oriented $x$-axis, the group of translations that fix the $x$-axis, and the group that is generated by the vector $\langle 2L, 0\rangle$.  The set of translates of the $y$-axis by $\mathcal H$ is the grating $G(L)$.

Compress notation by writing $\Lambda(z)$ instead of $\Lambda_{(0,0)}(z)$ and by writing $I(X)$ instead of $I_{(0,0)}(X)$.  Since reflection across the $y$-axis is an isometry,
\begin{equation}\label{eq:3:planeprob}
	P(I(X)=1) = \frac{1}{L}\int_{0}^{L}\frac{\Lambda(z)}{\pi L}\,{\rm d}z.
\end{equation}
 The path $\gamma$ that is given by 
\begin{align*}
	\gamma(t) = (z+ L\cos(t),L\sin(t)), \quad \text{with}\quad \frac{\pi}{2}+\arcsin\big(\tfrac{z}{L}\big) \leq t \leq \frac{3\pi}{2} - \arcsin\big(\tfrac{z}{L}\big),
\end{align*} parameterizes the arc of $C(z)$ with arc length $\Lambda(z)$ and endpoints given by the points of intersection of $C(z)$ with the $y$-axis, and so
\begin{align}\label{1.7}\notag P(I(X) = 1) = 1-\frac{2}{\pi L}\int_0^L\arcsin\!\big(\tfrac{z}{L}\big)\,{\rm d}z = \frac{2}{\pi}.
\end{align}

\subsection{The needle in the sphere}

View $\mathds S_r$ as an embedded Riemannian submanifold of $\mathds R^3$ and, to simplify computations, take its center to be $(0,0,0)$.  The distance between any two points $p$ and $q$ in $\mathds S_r$ is the geodesic distance.  Fix the length $2L$ to be equal to $\frac{\pi r}{n}$ for some natural number $n$ that is greater than 1. The isometry group, $\mathcal G$, of $\mathds S_r$ is the orthogonal group $O(3)$ and the geodesics of $\mathds S_r$ are the great circles.  Take $\mathcal{E}$ to be the counterclockwise-oriented great circle that is formed by the intersection of $\mathds S_r$ with the $(x,y)$-plane. Take $G(L)$ to be the set of all great circles in $\mathds S_r$ that intersect $(0,0,r)$ so that \[G(L)\cap\mathcal{E} = \Big\{\big(r\cos\!\big(\tfrac{2mL}{r}\big), r\sin\!\big(\tfrac{2mL}{r}\big), 0\big)\colon m \in \{0, \dots, 2n-1\}\Big\}.\]  Great circles in $G(L)$ intersect $\mathcal E$ at right angles at two antipodal points.  The set of rotations with the common axis given by the line that passes through $(0,0,r)$ and $(0,0,-r)$ is the subgroup $\mathcal G^\prime$ of $\mathcal G$, and $\mathcal H$ is the subgroup of $\mathcal G^\prime$ that is generated by the rotation $h_0$, where \[h_0(r, 0, 0) = \Big(r\cos\!\big(\tfrac{2L}{r}\big), r\sin\!\big(\tfrac{2L}{r}\big), 0\Big).\]  For any $w$ in $[-L, L]$, write $\Lambda(w)$ and $I(X)$ instead of $\Lambda_{(r, 0,0)}(w)$ and $I_{(r, 0,0)}(X)$.

\begin{proposition}\label{3:Prop:NSphere}
For any positive real number $r$, \begin{equation}\label{SphereProb}P(I(X)=1)=1-\frac{2}{\pi L} \int_{0}^{L}\arcsin\!\big(\!\tan\!\big(\tfrac{w}{r}\big)\cot(\cll)\big)\,{\rm d}w.\end{equation}
\end{proposition}

\begin{proof}
The circumference of any geodesic circle of radius $L$ in $\mathds S_r$ is $2\pi r\sin\!\big(\tfrac{L}{r}\big)$.  Since reflection across the $y$-axis is an isometry,  %
\eqref{1.3} implies that
\begin{equation}\label{AspherewA}P(I(X)=1)= \int_{0}^{L} \frac{\Lambda(w)}{\pi r\sin\!\big(\cll\big)} \cdot \frac{1}{L} \,{\rm d}w.\end{equation}

Take $G_0$ to be the great circle of $G(L)$ that intersects $(r,0,0)$, the set \[G_0 = \big\{(r\sin(\phi),0,r\cos(\phi))\colon 0 \leq \phi \leq 2\pi\big\}.\] The geodesic circle $C_L(w)$ of radius $L$ with center $\big(r\cos\!\big(\w\big), r\sin\!\big(\w\big), 0\big)$ is the set 
\begin{align}C_L(w) &= \Big\{\big(-r\sin\!\big(\w\big)\sin\!\big(\theta\big)\sin\!\big(\cll\big) + r\cos\!\big(\w\big)\cos\!\big(\cll\big),\\&\hspace{.25in} r\cos\!\big(\w\big)\sin(\theta)\sin\!\big(\cll\big) + r\sin\!\big(\w\big)\cos\!\big(\cll\big), -r\cos(\theta)\sin\!\big(\cll\big)\big)\colon  0\leq \theta\leq 2\pi\Big\}.\notag
\end{align}%
The points of intersection are therefore given by the equation
\[\sin\!\big(\theta\big) = -\tan\!\big(\w\big)\cot\!\big(\cll\big).\]
For any $w$ in $[0, L]$, a counterclockwise oriented path that parameterizes the smaller arc of $C_L(w)$ that the intersection with $G_0$ determines has endpoints given by the angles $\theta_1$ and $\theta_2$, where \[\theta_1=\pi+\arcsin\!\big(\!\tan\!\big(\w\big)\cot\!\big(\cll\big)\big) \quad \text{and}\quad \theta_2=2\pi-\arcsin\!\big(\!\tan\!\big(\w\big)\cot\!\big(\cll\big)\big),\] and so \begin{equation}\label{Asphere}\Lambda(w)=\pi r\sin\!\big(\cll\big)-2r\sin\!\big(\cll\big)\arcsin\!\big(\!\tan\!\big(\tfrac{z}{r}\big)\cot\!\big(\cll\big)\big).\end{equation} Equations \eqref{AspherewA} and \eqref{Asphere} together imply that 
\begin{align}\label{BuffonProbForSphere}
P(I(X)=1)%
=1-\frac{2}{\pi L} \int_{0}^{L}\arcsin\!\big(\!\tan\!\big(\tfrac{w}{r}\big)\cot\!\big(\cll\big)\big)\,{\rm d}w.
\end{align}
\end{proof}


\subsection{The geometry of the Poincar\'{e} disk}

View $\mathds H_k$ as the scaled Poincar\'{e} disk, the open unit disk $\mathds D$ in $\mathds R^2$ endowed with the Riemannian metric ${\rm d}s^2$, where
\begin{equation}\label{eq:poin:met}
{\rm d}s^2 = \frac{4k^2}{(1-(x^2+y^2))^2}({\rm d}x\otimes {\rm d}x + {\rm d}y \otimes {\rm d}y).
\end{equation}
Any geodesic of $\mathds H_k$ that contains $(0,0)$ is the intersection of a line in $\mathds R^2$ with $\mathds D$.  All other geodesics are the intersection with $\mathds D$ of a circle in $\mathds R^2$ that intersects the unit circle at right angles. 
Take $\mathcal{E}$ to be the geodesic given by the intersection of $\mathds D$ with the $x$-axis.  

For any point $p$ in $\mathds D$,
if the Euclidean distance from $p$ to $(0,0)$ is $d$, then the hyperbolic distance is $2k\tanh^{-1}(d)$.
The isometry group $\mathcal G$ of $\mathds H_k$ is the subgroup of the M\"{o}bius transformations that map the open unit disk to itself. %
The subgroup $\mathcal G^\prime$ of $\mathcal G$ that maps $\mathcal E$ to itself and preserves the orientation of $\mathcal E$ is the group of transformations of the form $F_\sigma$, where for any real number $\sigma$, 
\[
F_\sigma(x,y)= \left(\frac{\tau(x^2 + y^2) + (\tau^2 + 1)x + \tau}{(\tau x + 1)^2 + \tau^2y^2}, \frac{(1 - \tau^2)y}{(\tau x + 1)^2 + \tau^2y^2}\right),
\quad \text{where} \quad \tau = \tanh\!\big(\tfrac{\sigma}{2k}\big).
\]
Notice that for any real numbers $h$ and $\sigma$,
\begin{align}\notag
	F_\sigma\big(\!\tanh\!\big(\tfrac{h}{2k}\big), 0\big) = \big(\!\tanh\!\big(\tfrac{h + \sigma}{2k}\big), 0\big),
\end{align}
and so $F_\sigma$ is the M\"obius transformation that maps $\mathcal E$ to $\mathcal E$, preserves the orientation of $\mathcal E$, and moves points in $\mathcal E$ a signed hyperbolic distance of $\sigma$ along $\mathcal E$.  Take $\mathcal H$ to be the subgroup of $\mathcal G^\prime$ that is generated by the transformation $F_{2L}$.

For each $\sigma$ in $\mathds R$, take $G_\sigma$ to be the geodesic $F_\sigma(G_0)$, where $G_0$ is the unbounded geodesic in $\mathds H_k$ that is given by the intersection of the $y$-axis with $\mathds D$. %
Each $G_\sigma$ intersects $\mathcal E$ at a right angle at $\big(\tanh\!\big(\tfrac{\sigma}{2k}\big), 0\big)$. %
Take $G(L)$ to be the grating %
\begin{equation}\notag
G(L) = \{G_{2nL}\colon n \in \mathds{Z}\}.
\end{equation}%

Hyperbolic circles in $\mathds D$ are also Euclidean circles, but the hyperbolic radius and center and the Euclidean radius and center may differ.  An apropos example is the hyperbolic circle with hyperbolic radius $\lambda$ whose center is a point in $\mathcal E$ that is a signed distance of $h$ from $(0,0)$.  The Euclidean center $(x, 0)$ and radius $r$ of this circle are given by
\[x = \tfrac{1}{2}\Big(\!\tanh\!\big(\tfrac{h+\lambda}{2k}\big) + \tanh\!\big(\tfrac{h-\lambda}{2k}\big)\Big) \quad{\rm and}\quad r = \tfrac{1}{2}\Big(\!\tanh\!\big(\tfrac{h+\lambda}{2k}\big) - \tanh\!\big(\tfrac{h-\lambda}{2k}\big)\Big).\]

\subsection{The needle in the Poincar\'{e} disk}
For each $z$ in $\mathds R$, once again compress notation by writing $\Lambda(z)$ instead of $\Lambda_{(0,0)}(z)$ and $I(X)$ instead of $I_{(0,0)}(X)$. 

\begin{theorem}\label{PoinProb:A}
For any $(x,0)$ in $\mathcal E$,
\begin{multline}\label{EQ:PoinProb:A}
P(I_{(x,0)}(X)=1) =2
\left(1-\frac{1}{\pi L\sinh\!\big(\frac{L}{k}\big)}\int_{0}^{L}\frac{2\left(\tanh\!\big(\frac{z+L}{2k}\big)-\tanh\!\big(\frac{z-L}{2k}\big)\right)}{\sqrt{\left(1-\tanh^2\!\big(\frac{z+L}{2k}\big)\right)\left(1-\tanh^2\!\big(\frac{z-L}{2k}\big)\right)}}\right.\\ 
\left.\cdot\arctan\left(\sqrt{\tfrac{1-\tanh^2\left(\frac{z-L}{2k}\right)}{1-\tanh^2\left(\frac{z+L}{2k}\right)}}\cdot\sqrt{\tfrac{\tanh\left(\frac{L+z}{2k}\right)}{\tanh\left(\frac{L-z}{2k}\right)}}\right)\,{\rm d}z\right).
\end{multline}
\end{theorem}

\begin{proof}
It is convenient to take $x$ to be equal to $0$ and calculate the probability $P(I(X)=1)$.  The circumference of any geodesic circle of hyperbolic radius $L$ in $\mathds H_k$ is $2\pi k\sinh\big(\frac{L}{k}\big)$.  Since reflection across the $y$-axis is an isometry, \eqref{1.3} implies that
\begin{equation}\label{PspherewA}P(I(X)=1) = \int_{0}^{L} \frac{\Lambda(z)}{\pi k \sinh\!\big(\tfrac{L}{k}\big)}\cdot\frac{1}{L}\,{\rm d}z.\end{equation}

For any $z$ in $[0, L]$, the Euclidean center of $C_L(z)$ is $(x(z),0)$ and the Euclidean radius of $C_L(z)$ is $r(z)$ for some real numbers $x(z)$ and $r(z)$. Take $(0,\pm y)$ to be the two intersections of $C_L(z)$ with $G_0$ so that \[y = \sqrt{r(z)^2-x(z)^2}.\] 

The angle, $\theta(z)$, that is formed by the positive $x$-axis and the ray from $(0,0)$ to the point $(0,y)$ has the property that \begin{equation}\label{4:TanCalc}\cos(\theta(z)) = -\frac{x(z)}{r(z)} \quad \text{and}\quad \sin(\theta(z)) = \frac{\sqrt{r(z)^2 - x(z)^2}}{r(z)}.\end{equation} Equation~\eqref{eq:poin:met} and the fact that reflection across the $x$-axis is an isometry of $\mathds H_k$ together imply that the hyperbolic arc length $\Lambda(z)$ of the arc $\mathcal A$ of $C_L(z)$ that lies to the left of $G_0$ is given by%
\begin{align*}
\Lambda(z) = \int_{\mathcal A}\, {\rm d}s
=4kr(z)\int_{\theta(z)}^{\pi}\frac{{\rm d}t}{1-(x(z)+r(z)\cos(t))^2-(r(z)\sin(t))^2}.
\end{align*}
Rewrite the integral using the functions $A$ and $B$ that are given by 
\[A(z)=1-x(z)^2-r(z)^2\quad\text{and}\quad B(z)=2r(z)x(z),\] and integrate to obtain the equality
\begin{align*}
\Lambda(z)%
&=2k\!\left(\pi \sinh\!\big(\tfrac{L}{k}\big)-\frac{4r(z)}{\sqrt{A(z)^2-B(z)^2}}\arctan\!\left(\sqrt{\tfrac{A(z)+B(z)}{A(z)-B(z)}}\cdot\tan\!\big(\tfrac{\theta(z)}{2}\big)\right)\right).
\end{align*}
The equality
\[\tan\!\big(\tfrac{\theta(z)}{2}\big)=\frac{1-\cos(\theta(z))}{\sin(\theta(z))},\]
together with \eqref{4:TanCalc} and the equalities
\[x(z)-r(z)=\tanh\!\big(\tfrac{z-L}{2k}\big)\quad\text{and}\quad x(z)+r(z)=\tanh\!\big(\tfrac{z+L}{2k}\big),\]%
%
implies after some simplification that
\begin{multline*}
\Lambda(z) = 2k
\left(\pi \sinh\!\big(\tfrac{L}{k}\big)-\frac{2\left(\tanh\!\big(\tfrac{z+L}{2k}\big)-\tanh\!\big(\tfrac{z-L}{2k}\big)\right)}{\sqrt{\left(1-\tanh^2\!\big(\tfrac{z+L}{2k}\big)\right)\left(1-\tanh^2\!\big(\tfrac{z-L}{2k}\big)\right)}} \right.\\ %
\left.\cdot\arctan\left(\sqrt{\tfrac{1-\tanh^{2}\!\big(\tfrac{z-L}{2k}\big)}{1-\tanh^{2}\!\big(\tfrac{z+L}{2k}\big)}}\cdot\sqrt{\tfrac{\tanh\!\big(\tfrac{L+z}{2k}\big)}{\tanh\!\big(\tfrac{L-z}{2k}\big)}}\right)\right).
\end{multline*}
The formula for $\Lambda(z)$ together with \eqref{PspherewA} implies \eqref{EQ:PoinProb:A}.
\end{proof}


\section{Limiting Behavior}  

This section presents second order expansions in the parameter $L$ for the probabilities given by \eqref{SphereProb} and \eqref{EQ:PoinProb:A}.  Zeroth-order terms give the probability in the planar case (which is independent of $L$), first-order terms vanish, and second-order terms depend on the Gaussian curvature of the space.  The following subsections utilize the standard `little-o' notation to simplify expressions, where $\Lo(L)$ describes the behavior of a function as $L$ tends to $0$.

\subsection{Order estimates for spheres}

\begin{proposition}\label{Prop:SPHERE}
For any positive real number $r$, if $M$ is $\mathds S_r$, then \[P(I(X)=1)= \frac{2}{\pi} +\frac{2}{9\pi r^2}L^2 + {\rm o}(L^2).\]
\end{proposition}

\begin{proof}
Take $F$ to be given by \begin{equation}\label{sphereo0}F(L) = \frac{\pi}{2}\left(1-P(I(X)=1)\right).\end{equation}  Change variables and use the series expansions of $\tan$ and $\cot$ together with Proposition~\ref{3:Prop:NSphere} to obtain the expansion
\begin{equation}\label{sphereo1}F(L) = \int_{0}^{1}\arcsin\!\big(z + L^2\big(\tfrac{z^3}{3r^2}- \tfrac{z}{3r^2}\big) + {\rm o}(L^3)\big)\,{\rm d}z.\end{equation}
Although $F$ is not initially defined at $0$, it has a continuous extension to $0$.  Once again denote by $F$ this continuous extension.  Differentiate $F$ twice  and integrate over the variable $z$ to obtain the equalities%
\begin{equation}\label{sphereoo1}
F(0) = \frac{\pi}{2} -1, \quad
F^\prime(0) = 0,
\quad \text{and}\quad
F^{\prime\prime}(0) = %
-\frac{2}{9r^2},
\end{equation}
which imply that \begin{equation}\label{Fsphere}F(L) = \left(\frac{\pi}{2}-1\right) - \frac{L^2}{9r^2} + {\rm o}(L^2).\end{equation} This expansion of $F(L)$ to quadratic order together with \eqref{sphereo0} gives the desired expansion for the Buffon probability.
\end{proof}

\newcommand{\dell}{\ell\cdot}

\subsection{Order estimates for the Poincar\'{e} disk}

\begin{proposition}\label{prop:poincare}
For any positive real number $k$, if $M$ is $\mathds H_k$, then \[P(I(X)=1) =2\!\left(1-\frac{L}{\pi k\sinh\!\big(\frac{L}{k}\big)}\left(\left(\pi-1\right) + \frac{L^2}{6k^2}\left(\pi-\frac{1}{3}\right)+  {\rm o}(L^2)\right)\right).\]
\end{proposition}

\begin{proof}
Change variables to rewrite the equation in Theorem~\ref{PoinProb:A} as
\begin{multline}\label{PoinProb001}
P(I(X)=1) = 2
\Bigg(1-\frac{1}{\pi\sinh\!\big(\tfrac{L}{k}\big)}\int_{0}^{1}\tfrac{2\left(\tanh\left(\frac{(z+1)L}{2k}\right)-\tanh\left(\frac{(z-1)L}{2k}\right)\right)}{\sqrt{\left(1-\tanh^2\left(\frac{(z+1)L}{2k}\right)\right)\left(1-\tanh^2\left(\frac{(z-1)L}{2k}\right)\right)}}\\ %
\cdot\arctan\left(\sqrt{\tfrac{1-\tanh^2\left(\frac{(z-1)L}{2k}\right)}{1-\tanh^2\left(\frac{(z+1)L}{2k}\right)}}\cdot\sqrt{\tfrac{\tanh\left(\frac{(1+z)L}{2k}\right)}{\tanh\left(\frac{(1-z)L}{2k}\right)}}\right)\,{\rm d}z\Bigg).
\end{multline}
Take $H$ to be the function that is given by %
\begin{equation}\label{5:PoinH}
H(L) = \frac{\pi k\sinh\!\big(\tfrac{L}{k}\big)}{L}\left(1-\frac{1}{2}P(I(X)=1)\right),
\end{equation}
and use a series expansions for $\tan$ and the square root to obtain the equality %
\begin{multline}\label{Phhhhh}
H(L) =\int_0^1\left(1+ \frac{L^2}{6k^2}+{\rm o}(L^3)\right)\\\cdot\arctan\!\left(\left(1 + \tfrac{L^2z}{2k^2}+{\rm o}(L^3)\right)\sqrt{\tfrac{z+1}{1-z}\left(1-\tfrac{L^2z}{3k^2}+{\rm o}(L^4)\right)}\right)\,{\rm d}z. 
\end{multline}
Take $F$ to be the function that is given by
\begin{align}\label{PoinF}
F(L) &=\int_0^1\arctan\!\left(\left(1 + \tfrac{L^2z}{2k^2}+{\rm o}(L^3)\right)\sqrt{\tfrac{z+1}{1-z}\left(1-\tfrac{L^2z}{3k^2}+{\rm o}(L^4)\right)}\right)\,{\rm d}z\notag\\%
&=\int_0^1\arctan\!\left(\left(1 + \tfrac{L^2z}{3k^2}+{\rm o}(L^3)\right)\sqrt{\tfrac{z+1}{1-z}}\,\right)\,{\rm d}z.
\end{align}
The function $F$ has a continuous extension to $0$. Once again denote by $F$ this continuous extension.   Differentiate $F$ twice and integrate with respect to $z$ to obtain the equalities 
\begin{align*}%
F(0) = -\frac{1}{2}+\frac{\pi}{2},\quad F^\prime(0) = 0, \quad\text{and}\quad F^{\prime\prime}(0) &= \frac{1}{9k^2}.
\end{align*}
Expand $F(L)$ up to quadratic terms to obtain the equality %
\begin{equation}\label{PoinF2ord}
F(L) = -\frac{1}{2}+\frac{\pi}{2} + \frac{1}{18k^2}L^2 + {\rm o}(L^2),
\end{equation}
and use \eqref{5:PoinH}, \eqref{Phhhhh}, \eqref{PoinF}, and \eqref{PoinF2ord} to obtain the expansion up to quadratic terms in $L$ of the Buffon probability for $\mathds H_k$.
\end{proof}

\subsection{Buffon deficits and Gaussian curvature} Notice that the corollary to Theorem~\ref{Thm:Goal} follows immediately from Proposition~\ref{Prop:SPHERE} and Proposition~\ref{prop:poincare}.

\begin{proof}[Proof of Theorem~\ref{Thm:Goal}]
It is enough to establish the result for $M$ equal to $\mathds R^2$, $\mathds S_r$, or $\mathds H_k$.  The result is immediate in the planar setting.  For the $\mathds S_r$ setting, use Proposition~\ref{Prop:SPHERE} to obtain the equalities
\begin{align*}
\lim_{L\to 0^+}\frac{9\pi}{2}\frac{P(I(X)=1) - \frac{2}{\pi}}{L^2}& = \lim_{L\to 0^+}\frac{9\pi}{2}\frac{\frac{2}{\pi} +\frac{2}{9\pi r^2}L^2 + {\rm o}(L^2) - \frac{2}{\pi}}{L^2} = \frac{1}{r^2}.
\end{align*}
For the $\mathds H_k$ setting, use Proposition~\ref{prop:poincare} to obtain the equalities
\begin{align*}
&\lim_{L\to 0^+}\frac{9\pi}{2}\frac{P(I(X)=1) - \frac{2}{\pi}}{L^2}\\&\hspace{.35in} = \lim_{L\to 0^+}\frac{9\pi}{2}\frac{2
\left(1-\frac{L}{\pi k\sinh\!\big(\frac{L}{k}\big)}\left(\left(\pi-1\right) + \frac{L^2}{k^2}\left(\frac{\pi}{6}-\frac{1}{18}\right) +  {\rm o}(L^2)\right)\right) - \frac{2}{\pi}}{L^2}\\
&\hspace{.35in} = \frac{9\pi}{2}\left(-\frac{1}{3k^2} + \frac{1}{9\pi k^2}\right) + \frac{9\pi}{2}\lim_{L\to 0^+} \frac{2\left(1-\frac{L}{\pi k\sinh\!\big(\frac{L}{k}\big)}\left(\pi-1\right)\right)  - \frac{2}{\pi}}{L^2} = -\frac{1}{k^2}.
\end{align*}

\end{proof}

The similarity between Theorem~\ref{Thm:Goal} and the Bertrand-Diguet-Puiseux Theorem merits some further discussion.  Denote by ${\rm Circ}(C_L(M))$ and ${\rm Area}(C_L(M))$ the circumference and area of a geodesic circle of radius $L$ in $M$.  The Bertrand-Diguet-Puiseux Theorem and Theorem~\ref{Thm:Goal} together imply that \begin{align*}\kappa(M) & = \lim_{L\to 0^+}\frac{6}{L^2}\left(\frac{2\pi L - {\rm Circ}(C_L(M))}{2\pi L}\right)\\ &= \lim_{L\to 0^+}\frac{12}{L^2}\left(\frac{\pi L^2 - {\rm Area}(C_L(M))}{\pi L^2}\right)= \lim_{L\to 0^+}-\frac{9}{L^2}\left(\frac{\frac{2}{\pi} - P(I(X)=1)}{\frac{2}{\pi}}\right).\end{align*}  Although the various deficits (circumference, area, and Buffon) initially appear to involve $\kappa(M)$ in different ways, the ratios of the deficits to the planar quantities involve $\kappa(M)$ in the same way, as a second order term in the Taylor expansion of the ratio as a function of $L$.

\subsection*{Acknowledgements}

The current work grew from a collaboration with undergraduate students at the University of California, Riverside who participated in the undergraduate research program and the Faculty Led Education Abroad Program (FLEAP).  We thank Professor Yat Sun Poon, department chair (2015 -- 2021), who enthusiastically supported both programs, and Dr. Karolyn Andrews, for her support for Weisbart's FLEAP program that ran in summer 2018 and 2019.  We thank James Alcala, Christopher Kirchgraber, Frances Lam, and Alexander Wang for the time they spent as undergraduate researchers working on this project in its early stages.

We thank the anonymous reviewers for their interest in this paper and for their detailed comments and corrections.  Their input motivated substantial improvements to this work.


\end{document}